\documentclass[11pt,a4paper]{amsart}
\usepackage[utf8]{inputenc}
\usepackage[english]{babel}
\usepackage{amsmath,amsthm}
\usepackage{amsfonts}
\usepackage{amssymb}
\usepackage{amsrefs}
\usepackage{enumitem}
\usepackage{hyperref}
\usepackage{tikz}
\usetikzlibrary{arrows}
\usepackage{graphicx}
\usepackage[left=3cm,right=3cm,top=2cm,bottom=2cm]{geometry}

\newcommand{\relation}[3]{#1\xrightarrow{#3} #2}

\newcommand{\C}{\mathbb{C}}
\newcommand{\CC}[1]{\mathbb{C}^{#1}}

\newcommand{\SlC}{\mathrm{SL}_n(\mathbb{C})}
\newcommand{\SpC}{\mathrm{Sp}_{2n}(\mathbb{C})}

\newcommand{\Eu}[1]{\begin{pmatrix}
I_n&#1\\0&I_n
\end{pmatrix}}

\newcommand{\CwO}{\mathbb{C}^{2n}\setminus \{0\}}

\newcommand{\deriv}[1]{\frac{\partial}{\partial {#1}}}
\newcommand{\Deriv}[2]{\deriv{{#1}_{#2}}}

\newcommand{\ringofholo}{\mathcal{O}(X)}
\newcommand{\holiealg}{\mathrm{VF}(X)}
\newcommand{\holiealgint}{\mathrm{VF}^{int}(X)}
\newcommand{\algliealg}{\mathrm{VF}_{alg}(X)}
\newcommand{\algliealgint}{\mathrm{VF}_{alg}^{int}(X)}
\newcommand{\regularfunc}{\mathbb{C}[X]}
\newcommand{\Lie}{\mathrm{Lie}(X)}
\newcommand{\Liealg}{\mathrm{Lie}_{alg}(X)}

\newtheorem{theorem}{Theorem}[section]
\newtheorem{lemma}[theorem]{Lemma}
\newtheorem{cor}[theorem]{Corollary}
\newtheorem{prop}[theorem]{Proposition}
\newtheorem{defi}[theorem]{Definition}

\newtheorem{thm}{Theorem}

\title{A Criterion for the Density Property of Stein manifolds}
\author[R.B. Andrist]{Rafael B. Andrist}
\address{Department of Mathematics \\
American University of Beirut \\
Beirut, Lebanon \\ And Faculty of Mathematics and Physics \\
University of Ljubljana \\
Ljubljana, Slovenia}
\email{rafael-benedikt.andrist@fmf.uni-lj.si}

\author[G. Freudenburg]{Gene Freudenburg}
\address{Department of Mathematics \\
Western Michigan University}
\email{gene.freudenburg@wmich.edu}

\author[G. Huang]{Gaofeng Huang}
\author[F. Kutzschebauch]{Frank Kutzschebauch}
\author[J. Schott]{Josua Schott}
\address{Mathematical Institute \\ University of Bern \\  
Bern, Switzerland}
\email{gaofeng.huang@unibe.ch, josua.schott@unibe.ch, frank.kutzschebauch@unibe.ch}

\thanks{The research of Rafael Andrist was supported by the European Union (ERC Advanced grant HPDR, 101053085 to Franc Forstneri\v{c}) and grant N1-0237 from ARRS, Republic of Slovenia. Gaofeng Huang, Frank Kutzschebauch, and Josua Schott were supported by Schweizerischer Nationalfonds - Grant 200021-207335. }

\begin{document}


\maketitle
\tableofcontents
\begin{abstract}
    We generalize a criterion for the density property of Stein manifolds. As an application, we give a new, simple proof of the fact that the Danielewski surfaces have the algebraic density property. Furthermore, we have found new examples of Stein manifolds with the density property.
\end{abstract}

\section{Introduction}
We say that a Stein manifold $X$ has the \textit{density property}, if the Lie algebra generated by the $\C$-complete holomorphic vector fields on $X$ is dense in the Lie algebra of all holomorphic vector fields on $X$.
The density property is a precise way to require that the holomorphic automorphism group of a Stein manifold has to be very large. In particular, the so-called Andersén--Lempert Theorem is valid for any Stein manifold $X$ with the density property. It provides a Runge-type approximation theorem for holomorphic deformations of a domain $\Omega \subseteq X$ by holomorphic automorphisms of $X$.
For an overview of the density property and its wealth of applications in solving holomorphic problems of geometric nature we refer to the overview articles of the fourth author with Kaliman \cite{KK:onthepresent}, with Forstneri\v c \cite{FoK:thefirst} and to his chapter in \cite{K:manifoldswith}.

The notion of the density property was introduced by Varolin, who himself \cite{Varolin} and in collaboration with T\'{o}th \cites{VT:holomorphicDiffeos, VT:holomorphicDiffeosSemisimple} came up with the first classes of manifolds admitting the density property: $\C^n$ for $n\geq 2$, and homogeneous spaces of semi-simple complex Lie groups with trivial center. 

In \cite{KalimanKutzschebauch} Kaliman and the fourth author worked out a strategy to find sufficient conditions whether a given Stein manifold $X$ has the density property. This enabled them to enlarge the classes of manifolds with density property substantially \cites{FoK:thefirst, KK:densitypropertyand volume}: all homogeneous spaces of linear algebraic groups except $\C$ and $(\C^*)^n$. Roughly speaking, the idea is based on finding sufficiently many so-called \textit{compatible pairs} of $\C$-complete holomorphic vector fields.
They also established the algebraic density property for Danielewski surfaces \cite{KalKutHyperSurface}*{Theorem 1} despite the lack of compatible pairs in this particular case. 

In this article, we generalize the notion of the compatible pair and the Kaliman--Kutzschebauch criterion. Given $\C$-complete holomorphic vector fields $\theta_1,\dots,\theta_n$ on $X$ we call the ordered $n$-tuple $(\theta_1,\dots,\theta_n)$ $n$-compatible if the following two conditions are satisfied:
\begin{enumerate}[label=(\arabic*)]
    \item There exists a non-zero ideal $I\subset \ringofholo$ in the ring of holomorphic functions on $X$ such that
    \[ I \subset \overline{\mathrm{span}\left( \prod_{i=1}^n \ker \theta_i\right)} \]
    \item There is a graph $(G,\pi,\epsilon)$, where
    \begin{enumerate}[label=(\roman*)]
        \item $G$ is a rooted directed tree with orientation towards the root
        \item $\pi\colon\mathrm{Vert}(G)\to \{\theta_1,\dots,\theta_n\}$ is a bijection with $\pi(\mathrm{root})=\theta_1$
        \item $\epsilon\colon \mathrm{Edges}(G)\to \ringofholo$ is a mapping with
        $$ \epsilon(v,w)\in \left(\ker\pi(v)^2\setminus \ker\pi(v)\right)\cap \ker\pi(w). $$
    \end{enumerate}
\end{enumerate}
The above-mentioned Danielewski surfaces do not admit a compatible pair, but they admit a compatible 3-tuple according to our new definition. As we will see below, this enables us to give a more concise proof of the density property for Danielewski surfaces. Moreover, we have discovered a manifold, which we call a Gromov--Vaserstein fiber, for which we do not know whether there exist compatible pairs or not, but which turns out to have the density property thanks to the existence of a compatible 3-tuple.

Let $\mathrm{Aut}_{hol}(X)$ denote the group of holomorphic automorphisms on $X$. Then the first main result can be formulated as follows.
\begin{thm}[Generalized Kaliman--Kutzschebauch criterion]\label{theorem:KalKut}
Let $X$ be a homogeneous Stein mani\-fold with respect to $\mathrm{Aut}_{hol}(X)$. Assume there exist $N=\dim(X)$ compatible $n$-tuples $\{(\theta_{1,i},\dots,\theta_{n,i})\}_{i=1}^N$ such that the vectors $(\theta_{1,1})_x,\dots,(\theta_{1,N})_x$ span the tangent space $T_xX$ at some point $x\in X$. Then $X$ has the density property.    
\end{thm}

In practice it could be difficult to find $\dim(X)$ compatible $n$-tuples. The next result transforms this condition into another form for an application of Theorem \ref{theorem:KalKut}.
\begin{thm}\label{theorem:SufficientConditionsDP}
    Let $X$ be a Stein manifold and $(\theta_1,\dots,\theta_n)$ a compatible $n$-tuple. Assume that there exist $\C$-complete holomorphic vector fields $V_1,\dots,V_N$ on $X$, $N\geq \dim(X)$, which span the tangent bundle $TX$. If there exist functions $f_i\in \ker V_i$ such that $f_i(x)=0$ and $d_{x}f_i(\theta_1)\neq 0$ for some point $x\in X$ and $i = 1, \dots, N$, then $X$ has the density property.
\end{thm}

The reason for us to develop this new criterion is to prove the density property for a certain class of
affine algebraic manifolds, naturally arising in the solution to the factorization problems of holomorphic matrices into elementary factors, namely the solutions to the Gromov--Vaserstein problems for the special linear groups by Ivarsson and Kutzschebauch \cite{SLnPaper}, and for the symplectic groups by Schott \cite{Schott}.  These manifolds are the fibers of certain polynomial mappings $P\colon\C^m\to \C^k$, we call them  Gromov--Vaserstein fibers.
They share the property that each smooth fiber $P^{-1}(y)$ is biholomorphic to a product $\mathcal{G}\times \C^M$ for some holomorphically flexible Stein manifold $\mathcal{G}$ and some natural number $M$. In \cite{SLnPaper} and \cite{Schott} it is shown that $\mathcal{G}$ is elliptic in the sense of Gromov and thus an Oka manifold. For the notion of Oka manifolds see \cite{ForstnericBook}*{Chapter 5}. Our main goal is to show that $\mathcal{G}$ has the density property which is much stronger than being Oka. 

Note that it is much easier to prove that the Gromov--Vaserstein fibers $\mathcal{G}\times \C^M$ have the density property for $M\geq 1$. In fact, an application of the original Kaliman--Kutzschebauch criterion shows that products of holomorphically flexible manifolds with affine spaces have the density property. Since this result is known in the literature only in a very restricted form (Varolin \cite{Varolin} proved that $G\times \C$ has the density property when $G$ is a Stein complex Lie group) we include it here. 

Recall that manifold $X$ is called (\textit{holomorphically}) \textit{flexible} at a point $x\in X$ if the values at $x$ of $\C$-complete holomorphic vector fields on $X$ span the tangent space $T_xX$. The manifold $X$ is \textit{flexible} if it is flexible at every point $x\in X$.

\begin{thm}\label{theorem:holFlexDP}
    Let $X$ be a holomorphically flexible Stein manifold. Then $X\times \C$ has the density property.
\end{thm}

The reader should compare this theorem to the result of Ugolini and Winkelmann \cite{UgoWink}, where they prove that the total space of a line bundle $\pi\colon E\to X$ over a Stein manifold $X$ with the density property has again the density property, if there exists a so-called $\pi$-incompatible holomorphic automorphism of $E$.

 As mentioned above, the manifolds $\mathcal{G}$ from the Gromov--Vaserstein fiber $\mathcal{G} \times \C^M$
 are our new examples of manifolds with the density property. We shall encounter an interesting special case,  where $\mathcal{G}$ is given by
\[ \mathcal{G}=\left\lbrace (z_2,z_3,w_1,w_2,w_3)\in \C^5: \begin{pmatrix}
    w_1 & w_2 \\ w_2 & w_3
\end{pmatrix} \begin{pmatrix}
    z_2 \\ z_3
\end{pmatrix} = \begin{pmatrix}
    b_1\\ b_2
\end{pmatrix}\right\rbrace\]
with $(b_1,b_2)\in \C^2\setminus \{(0,0)\}.$ We do not know whether there exist compatible pairs on $\mathcal{G}$ or not. However, there is a compatible $3$-tuple and, moreover, there are $\C$-complete holomorphic vector fields $V_1,V_2$ and $V_3$ on $\mathcal{G}$ which satisfy the conditions of Theorem \ref{theorem:SufficientConditionsDP}. Hence $\mathcal{G}$ has the density property.

\bigskip
 
 This paper is organised as follows. In Section 2, we prove the algebraic version of Theorem \ref{theorem:KalKut}. The same proof works in the holomorphic setting after minor adjustments. The basic idea is to find a suitable $\ringofholo$-submodule $L$ of $\holiealg$, the Lie algebra of all holomorphic vector fields on $X$, such that $L$ is contained in the closure $\overline{\Lie}$ of the Lie algebra $\Lie$ generated by the $\C$-complete holomorphic vector fields on $X$. 

 In Section 3, we turn our attention to the proof of  Theorems \ref{theorem:SufficientConditionsDP} and \ref{theorem:holFlexDP}.  
 
In Section 4, we consider applications of Theorems \ref{theorem:KalKut} and \ref{theorem:SufficientConditionsDP}. First, we give a new proof for the fact that the Danielewski surfaces have the algebraic density property. And then, we prove the density property for the manifolds arising in the Gromov--Vaserstein fibration. These manifolds are 
in the symplectic case given as the common zero set of an arbitrary big number of polynomial equations. This makes it impossible to prove the density property by direct calculation as in the
groundbreaking works of Anders\'en--Lempert \cite{AL} and Varolin \cite{Varolin}.


\section{Generalization of the Kaliman--Kutzschebauch criterion}
We will prove the algebraic version of Theorem \ref{theorem:KalKut}. Let $X$ be an algebraic manifold. We let $\algliealg$ denote the Lie algebra of all algebraic vector fields on $X$ and $\algliealgint$ the set of $\C$-complete algebraic vector fields on $X$. Let $\Liealg$ denote the Lie algebra generated by $\algliealgint$.

\begin{defi}[\cite{Varolin}]
    An algebraic manifold $X$ has the algebraic density property if $$\Liealg = \algliealg.$$
\end{defi}

For an algebraic manifold $X$ we let $\mathbb{C}[X]$ denote the algebra of regular functions on $X$ and $\mathrm{Aut}_{alg}(X)$ the group of all algebraic automorphisms of $X$.
\begin{defi}[\cite{KalimanKutzschebauch}]
    Let $X$ be an algebraic manifold and $x_0\in X$. A finite subset $M$ of the tangent space $T_{x_0}X$ is called a generating set if the image of $M$ under the action of the isotropy subgroup of $x_0$ (in $\mathrm{Aut}_{alg}(X)$) spans $T_{x_0}X$ as a complex vector space. 
\end{defi}

Kaliman and Kutzschebauch \cite{KalimanKutzschebauch}*{Theorem 1} proved the following
\begin{theorem}\label{theorem:KalimanKutzschebauch}
    Let $X$ be an affine algebraic manifold that is $\mathrm{Aut}_{alg}(X)$-homogeneous. Let $L$ be a $\mathbb{C}[X]$-submodule of $\algliealg$ such that $L\subset \Liealg$. If the fiber $L_{x_0}=\{V_{x_0}: V\in L\}\subset T_{x_0}X$ over some point $x_0\in X$ contains a generating set, then $X$ has the algebraic density property.
\end{theorem}

A similar result is true in the holomorphic setting. For a Stein manifold $X$, let $\holiealg$ denote the Lie algebra of all holomorphic vector fields on $X$ and $\Lie$ the Lie algebra generated by the complete holomorphic vector fields on $X$. 
\begin{theorem}\label{theorem:KalKutHolomorph}
    Let $X$ be a Stein manifold that is homogeneous with respect to $\mathrm{Aut}_{hol}(X)$ the group of holomorphic automorphisms on $X$. Let $L$ be a $\ringofholo$-submodule of $\holiealg$ such that $L$ is contained in the closure of $\Lie$. If the fiber $L_{x_0}=\{V_{x_0}: V\in L\}\subset T_{x_0}X$ over some point $x_0\in X$ contains a generating set, then $X$ has the density property.
\end{theorem}

Kaliman and Kutzschebauch provide a strategy to construct such a submodule using so-called \textit{compatible pairs}. We will now generalize this idea, and for that we need some preparation.

\begin{defi}
    Let $S\subset \algliealg$ be a finite non-empty set. An $S$-admissible graph is a triple $(G,\pi,\epsilon)$, where
    \begin{enumerate}[label=(\roman*)]
        \item $G$ is a directed graph
        \item $\pi\colon\mathrm{Vert}(G)\to S$ is a bijective map
        \item $\epsilon\colon\mathrm{Edge}(G)\to \regularfunc$ such that
        $$ \epsilon(u,v) \in \left(\ker \pi(v)^2\setminus \ker\pi(v) \right)\cap \ker\pi(w).$$
    \end{enumerate}
\end{defi}

\begin{defi}
    An ordered $n$-tuple $(\theta_1,\dots,\theta_n)$ of complete vector fields on $X$ is compatible if
    \begin{enumerate}[label=(\roman*)]
        \item there exists a finite rooted $S$-admissible tree $(G,\pi,\epsilon)$ with $\theta_1=\pi(\mathrm{root})$ with $S=\{\theta_1,\dots,\theta_n\}$,
        \item there is a non-zero ideal $I\subset \regularfunc$ with
        $$ I \subset \mathrm{span}\left( \prod_{i=1}^n \ker \theta_i\right).$$
    \end{enumerate}
\end{defi}

\noindent \textbf{Convention:}\\
Given $S\subset \algliealg$ and an $S$-admissible graph $(G,\pi,\epsilon)$, we write $\theta$ for both the vector field $\theta \in S$ and the vertex $\pi^{-1}(\theta)\in \mathrm{Vert}(G)$. Moreover, if $\epsilon(\theta,\phi) \in \left(\ker\theta^2\setminus \ker\theta\right)\cap \ker\phi$, we write $$\relation{\theta}{\phi}{a_{\theta}},$$ where $a_{\theta}:=\epsilon(\theta,\phi)$. And finally, we sometimes write $f^\theta$ for regular functions with $f^\theta\in \ker\theta$. \\

Using the compatible $n$-tuples, we prove the existence of a $\C[X]$-submodule. We formulate this important intermediate step in the proof of Theorem \ref{theorem:KalKut} as follows.
\begin{theorem}\label{theorem:submodule}
    Let $\theta_1,\dots,\theta_n$ be complete algebraic vector fields on an algebraic manifold $X$. If $(\theta_1,\dots,\theta_n)$ is a compatible $n$-tuple, then there exists a $\mathbb{C}[X]$-submodule $L$ of $\algliealg$ with ${L\subset \mathrm{Lie}_{alg}(X)}$.
\end{theorem}

We begin with an observation.
\begin{lemma}\label{lemma:LieAlgebra1}
Let $\theta$ be a complete vector field on $X$ and $\varphi\in \mathrm{Lie}_{alg}(X)$ such that $a\varphi\in \mathrm{Lie}_{alg}(X)$ for some $a\in\mathbb{C}[X]$ with $a\in \ker\theta$. Then we have
\begin{align*}
[a\theta,\varphi] - [\theta,a\varphi] = -\varphi(a)\cdot \theta \in \mathrm{Lie}_{alg}(X).
\end{align*}
\end{lemma}
\begin{proof}
For any $f\in \mathbb{C}[X]$ we get
\begin{align*}
[a\theta,\varphi](f) &= a\theta(\varphi(f)) - \varphi(a\theta(f)) \\&= a\theta(\varphi(f)) - \varphi(a)\cdot \theta(f) - a\varphi(\theta(f))
\end{align*}
and
\begin{align*}
[\theta,a\varphi](f) &= \theta(a\varphi(f)) - a\varphi(\theta(f)) \\&= \underbrace{\theta(a)}_{=0}\varphi(f) + a\theta(\varphi(f)) - a\varphi(\theta(f))\\&= a\theta(\varphi(f)) - a\varphi(\theta(f)).
\end{align*}
By assumption, $\varphi$ and $a\varphi$ are in $\Liealg$. Moreover, since $\theta$ is complete and $a\in \ker\theta$, the vector fields $\theta$ and $a\theta$ are also in $\Liealg$. Therefore
\begin{align*}
[a\theta,\varphi]-[\theta,a\varphi] 
&= a\theta\circ\varphi - \varphi(a)\cdot \theta - a\varphi\circ\theta - a\theta\circ\varphi + a\varphi\circ\theta\\
&=  - \varphi(a)\cdot \theta \in \Liealg
\end{align*}
and this proves the claim.
\end{proof}

\begin{lemma}
    Let $S=S_0\cup \{\theta\}$ be a finite set of vector fields on $X$ and $(T,\pi,\epsilon)$ an $S$-admissible rooted tree with $\pi(\mathrm{root})=\theta$. Assume there is a regular function $a_\theta\in \ker\theta^2\setminus \ker\theta$. Then $(T,\tilde{\pi},\epsilon)$ is $S'$-admissible, where $S'=S_0\cup \{a_\theta\theta\}$ and
    $$ \tilde{\pi}(v) = \begin{cases}
        \pi(v) & v\neq \mathrm{root},\\
        a_{\theta}\theta & v=\mathrm{root}.
    \end{cases}$$
\end{lemma}
\begin{proof}
    Clearly, $\tilde{\pi}\colon\mathrm{Vert}(T)\to S'$ is bijective. By definition of $\tilde{\pi}$, it remains to show that
    $$ \epsilon(v,\mathrm{root})\in \left(\ker\tilde{\pi}(v)^2\setminus \ker\tilde{\pi}(v)\right)\cap \ker(a_\theta\theta).$$
    And that is true because $\ker\theta = \ker(a_\theta\theta)$.
\end{proof}

Let $S$ be a finite non-empty set of vector fields on $X$ and $(T,\pi,\epsilon)$ an $S$-admissible rooted tree. For $\theta\in S$, let $T_\theta$ denote the subtree of $T$ with root $\pi^{-1}(\theta)$.

\begin{cor}\label{cor:admissSubtree}
    Let $S$ be a finite non-empty set of vector fields on $X$ and $(T,\pi,\epsilon)$ an $S$-admissible rooted tree.  Then $(T_\theta, \pi|_{\mathrm{Vert}(T_\theta)},\epsilon|_{\mathrm{Edge}(T_\theta)})$ is $\pi(\mathrm{Vert}(T_\theta))$-admissible.
\end{cor}

 A vertex $\theta\in \mathrm{Vert}(T)$ is called a \textit{leaf} if $\mathrm{Vert}(T_\theta)= \{\theta\}$.

\begin{prop}\label{prop:rootInLieAlgebra}
Let $S$ be a finite non-empty set of complete algebraic vector fields on an algebraic manifold $X$. Assume there is an $S$-admissible rooted tree $(T,\pi,\epsilon)$ with root $V$. Then
$$ \left( \prod_{\varphi\in \mathrm{Vert}(T)\setminus \{V\}} f_\varphi \varphi(a_\varphi)\right) f_V V \in \mathrm{Lie}_{alg}(X).$$
\end{prop}

\begin{proof}
The proof is by induction on the depth of the tree. Let's start with a leaf $\theta\in \mathrm{Vert}(T)$. Then the subtree $T_\theta$ consists of only one vertex and no edges. Hence
$$ \left( \prod_{\varphi\in \mathrm{Vert}(T_\theta)\setminus \{\theta\}} f_\varphi\varphi(a_\varphi)\right) f^\theta\theta = f^\theta \theta \in \mathrm{Lie}_{alg}(X),$$
since $f^\theta\theta$ is complete. This proves the base case.

Now we consider a vertex $\theta\in \mathrm{Vert}(T)$ which is not a leaf. By the induction hypothesis, the proposition is true for all vertices $\psi\in \mathrm{Vert}(T_\theta)$ with $\relation{\psi}{\theta}{{a_{\psi}}}$, that is, we have
$$ \tilde{\psi}:=\left( \prod_{\varphi\in \mathrm{Vert}(T_\psi)\setminus \{\psi\}} f^\varphi\varphi(a_\varphi)\right) f^\psi\psi \in \mathrm{Lie}_{alg}(X)$$
for all such vertices $\psi$. By Corollary \ref{cor:admissSubtree}, the subtree $T_\psi$ is $S':=\pi(\mathrm{Vert}(T_\psi))$-admissible with root $\psi$. Moreover, there is a regular function $a_\psi\in \left(\ker\psi^2\setminus \ker\psi\right)\cap \ker\theta$ by assumption. We can therefore apply Lemma 3.8 to conclude that the subtree $T_\psi$ is $S''$-admissible, where $S''$ is obtained by replacing $\psi$ with $a_\psi\psi$. By the induction hypothesis, we get
$$ a_\psi\tilde{\psi} \in \Liealg$$
for all vertices $\psi$ with $\relation{\psi}{\theta}{{a_{\psi}}}$.

In the next step, we let $\{\psi_1,\dots,\psi_N\}$ denote the set of vertices $\psi\in \mathrm{Vert}(T_\theta)$ such that $\relation{\psi}{\theta}{a_\psi}$. Define
$$ W_1 := [a_{\psi_1}\tilde{\psi}_1,f^\theta \theta] - [\tilde{\psi}_1,a_{\psi_1}f^\theta\theta] = \tilde{\psi}_1(a_{\psi_1})f^\theta\theta, $$
for an arbitrary regular function $f^\theta$ with $f^\theta\in \ker\theta$.
Observe that $\tilde{\psi}_1,a_{\psi_1}\tilde{\psi}_1\in \Liealg$ and $f^\theta\theta, a_{\psi_1}f^\theta\theta$ are even complete, since $f^\theta, a_{\psi_1}f^\theta\in \ker\theta$. Lemma \ref{lemma:LieAlgebra1} implies therefore $W_1\in \Liealg$ for every $f^\theta\in \ker\theta$. In particular, we have $fW_1\in \Liealg$ for all $f \in \ker\theta$. We continue and define 
$$ W_k:=[a_{\psi_k}\tilde{\psi}_k, W_{k-1}] - [\tilde{\psi}_k,a_{\psi_k}W_{k-1}]$$
for ${2\leq k\leq N}$. A repeated application of the same reasoning as for $W_1$ implies that
$$ W_k= \pm\tilde{\psi}_1(a_{\psi_1})\cdots \tilde{\psi}_k(a_{\psi_k})f^\theta \theta \in \Liealg,\quad 1\leq k\leq N.$$
It should be mentioned that we can ignore the sign without loss of generality, since $-1$ can be absorbed by $f^\theta$. In summary, we conclude that
\begin{align*}
W_N &= \left(\prod_{k=1}^N \tilde{\psi}_k(a_{\psi_k})\right) f^\theta\theta\\
&= \left( \prod_{k=1}^N f^{\psi_k}\psi_k(a_{\psi_k})\prod_{\varphi\in \mathrm{Vert}(T_{\psi_k})\setminus \{\psi_k\}} f^\varphi\varphi(a_\varphi)\right) f^\theta\theta\\
&= \left( \prod_{\varphi\in \mathrm{Vert}(T_\theta)\setminus\{\theta\}} f^\varphi \varphi(a_\varphi)\right) f^\theta\theta.
\end{align*}
Since we chose an arbitrary $\theta$, the proposition follows with $\theta=V$.
\end{proof}

\begin{proof}[Proof of Theorem \ref{theorem:submodule}]
Let $(\theta_1,\dots,\theta_n)$ be a compatible $n$-tuple and write $S=\{\theta_1,\dots,\theta_n\}.$ By assumption, there exists an $S$-admissible tree $(T,\pi,\epsilon)$ with $\pi(\mathrm{root})=\theta_1$ and a non-zero ideal $I \subset \mathrm{span}(\prod_{\theta\in \mathrm{Vert}(T)} \ker\theta)$. By Proposition \ref{prop:rootInLieAlgebra}, we have
$$ \left( \prod_{\varphi\in \mathrm{Vert}(T)\setminus \{\theta_1\}} f^\varphi \varphi(a_\varphi)\right) f^{\theta_1} \theta_1 \in \mathrm{Lie}_{alg}(X)$$
and hence
$$ L:=\prod_{\theta\in \mathrm{Vert}(T)\setminus\{\theta_1\}}\theta(a_\theta) I \theta_1$$
is the desired $\mathbb{C}[X]$-submodule.
\end{proof}

The following is a generalized version of \cite{KalimanKutzschebauch}*{Theorem 2}.
\begin{theorem}\label{theorem:generalizedTheorem2}

    Let $X$ be a homogeneous affine algebraic manifold with finitely many compatible $n$-tuples $\{(\theta_{k,1},\dots,\theta_{k,n})\}_{k=1}^m$ such that for some $x_0\in X$, $\{(\theta_{k,1})_{x_0}\}_{k=1}^m\subset T_{x_0}X$ is a generating set. Then $X$ has the algebraic density property.
\end{theorem}

\begin{proof}
    By Theorem \ref{theorem:submodule}, there exist non-zero ideals $I_k$ such that $\mathrm{Lie}_{alg}(X)$ contains the submodules $L_k=I_k\theta_{k,1}$. Hence, there exists a non-zero ideal $J\subset \mathbb{C}[X]$ such that $\mathrm{Lie}_{alg}(X)$ contains the submodule
    $$ L = \biggl\{\sum_{k=1}^m \alpha_k \theta_{k,1}: \alpha_1,\dots,\alpha_m\in J \biggl\}. $$
    Since $\{(\theta_{k,1})_{x_0}\}_{k=1}^m$ remains a generating set under small perturbations of the base point $x_0$, we can suppose that $x_0$ does not belong to the zero locus of $J$. For such $x_0$ the set $\{V_{x_0}: V\in L\}$ contains a generating set, and by Theorem \ref{theorem:KalimanKutzschebauch} it follows that $X$ has the algebraic density property.
\end{proof}

\section{Holomorphic flexibility and the density property}
In this section we prove Theorems  \ref{theorem:SufficientConditionsDP} and \ref{theorem:holFlexDP}. To do this, we pass from the algebraic to the holomorphic setting. Given a Stein manifold $X$, we let $\holiealg$ denote the Lie algebra of all holomorphic vector fields on $X$ and $\holiealgint$ the set of complete holomorphic vector fields on $X$. Moreover, $\Lie$ denotes the Lie algebra generated by $\holiealgint$ and $\ringofholo$ the ring of holomorphic functions on $X$. Recall the definition of the density property.

\begin{defi}[\cite{Varolin}]
    A Stein manifold $X$ has the density property, if
    $$ \overline{\Lie} = \holiealg,$$
    where the closure is with respect to the compact-open topology.
\end{defi}

As mentioned in the introduction, we introduce the notion of a (holomorphic) compatible $n$-tuple.
\begin{defi}\label{def:holomorphicCompatibleTuple}
    Given complete holomorphic vector fields $\theta_1,\dots,\theta_n$ on a Stein manifold $X$, the ordered $n$-tuple $(\theta_1,\dots,\theta_n)$ is $n$-compatible if the following two conditions are satisfied:
    \begin{enumerate}[label=(\arabic*)]
        \item There exists a non-zero ideal $I\subset \ringofholo$ such that
        $$ I\subset \overline{\mathrm{span}\left(\prod_{i=1}^n\ker \theta_i\right)}.$$
        \item There is a graph $(G,\pi,\epsilon)$, where
        \begin{enumerate}[label=(\roman*)]
            \item $G$ is a rooted directed tree with orientation towards the root
            \item $\pi\colon\mathrm{Vert}(G)\to \{\theta_1,\dots,\theta_n\}$ is a bijection with $\pi(\mathrm{root})=\theta_1$
            \item $\epsilon\colon\mathrm{Edges}(G)\to \ringofholo$ is a mapping with
            $$ \epsilon(v,w)\in \left(\ker\pi(v)^2\setminus \ker\pi(v)\right)\cap \ker\pi(w). $$
        \end{enumerate}
    \end{enumerate}
\end{defi}

The next lemma can be found in \cite{FlexibleVarieties}*{Lemma 4.1} and plays an important role in the proof of Theorem \ref{theorem:holFlexDP}.  
\begin{lemma} \label{calculation}
  Let $V$ be a complete vector field on a manifold $X$, $f\in \mathcal{O} (X)$ a function in the kernel of $V$ and $x \in X$
  a point with $f(x) = 0$. Denote the flow of the complete  field $f V$ at time $t$
   by $\varphi_t$ (fixing $x$ by our assumption). Then 
the differential $d_x \varphi_t$, which is an endomorphism of $T_x X$,  acts on a tangent vector $W \in T_x X$ as follows

$$W \mapsto W + t \cdot d_x f (W) \cdot V_x. $$
\end{lemma}
We are ready to prove Theorem \ref{theorem:holFlexDP}, which extends a result by Varolin, who proved it in the special case when $X$ is a Stein complex Lie group \cite{Varolin}.
\setcounter{thm}{2}
\begin{thm}\label{lemma:flexible} Suppose $X$ is a holomorphically flexible Stein  manifold. Then $X \times \C$ 
has the density property.
\end{thm}

\begin{proof} Let us denote the coordinate on $\C$  by $t$ and suppose that $V$ is a complete vector field on $X$. Then
$(V, \frac{\partial} {\partial t})$  is a compatible pair on $X \times \C$ (where we denote the obvious extension of $V$ to $X\times \C$ again by $V$).
Indeed, $\ker V$ contains all functions of $t$ and $\ker \frac{\partial} {\partial t}$ contains all functions on $X$
and the function $a\in \mathcal{O} (X\times \C)$ defined by 
$a := t$ is of degree 1 with respect to $\frac{\partial} {\partial t}$ and in $\ker V$.  If $V_i$ is a spanning set of complete holomorphic vector fields on $X$, the relevant vectors in the compatible pairs $(V_i, \frac{\partial} {\partial t})$ span $T_x X \subset T (X\times \C)$ at any point $x$. We create a new compatible pair by applying $\alpha^*$ to the pair $(V, \frac{\partial} {\partial t})$ for a suitable automorphism $\alpha \in \mathrm{Aut}_{hol} (X\times \C)$. We choose a point $x \in X$ where $V(x) \ne 0$ and a holomorphic function $f\in \mathcal{O} (X)$ with 
$f(x)=0$ and $d_x f (V) \ne 0$, which is possible on Stein manifolds by a standard application of Cartan's Theorem B. Now, let $\alpha$
be the time-one flow of the complete field $f \frac{\partial} {\partial t}$. By Lemma \ref{calculation} we have $\alpha^* (V) (x,t) =
V (x,t) + d_x f (V) \frac{\partial} {\partial t}$ and together with
other compatible pairs $(V_i, \frac{\partial} {\partial t})$  we have a spanning set of $T_{(x,0)} (X\times \C)$.

Adjoining $\frac{\partial} {\partial t}$ to any collection 
of spanning (the tangent space at each point) complete holomorphic vector fields on $X$ yields a collection of 
complete fields spanning on $X\times \C$, showing that $X\times \C$ is homogeneous under its holomorphic automorphism group. 

Now Theorem \ref{theorem:KalKutHolomorph} implies the claim. \qedhere
\end{proof}

\begin{lemma}\label{lemma:theorem2}
 Let $X$ be a holomorphically flexible Stein manifold and $(\theta_1,\dots,\theta_n)$ a compatible $n$-tuple. Assume that there are complete holomorphic fields $V_1, V_2, \ldots ,V_N$ which span the tangent space $T_x X$ at a point $x \in X$ and admit functions $f_i \in \ker V_i$ with $f_i (x) =0 $ and $d_x f_i (\theta_1) \ne 0$ for $i = 1, \dots, N$. Then
 there is a non-trivial $\ringofholo$-submodule $L $ of $\holiealg$ such that $L \subset \overline{\Lie}$ and the fiber $L_x=\{V_x: V\in L\}$ contains a generating set.
\end{lemma}
\begin{proof} Let $J \subset \ringofholo$ be the ideal of the compatible $n$-tuple $(\theta_1,\dots,\theta_n)$. We can assume that the point $x$ is not in the zero locus of $J$, since spanning the tangent space is an open condition. The holomorphic version of Theorem \ref{theorem:submodule} implies that the compatible $n$-tuple induces a holomorphic function $f \colon X\to \C$ such that the submodule $L:=J \cdot (f\theta_1)$ is contained in $\overline{\Lie}$.

Let $\varphi_t^i$ denote the flow of the complete holomorphic vector field $f_iV_i$. Then let $L_i=(\varphi^i_t)^*(L)$ be the submodules obtained by pulling back $L$ by $\varphi_t^i$.
Observe that the sum $\tilde{L}:=L_1+ \ldots +L_N$ is a submodule contained in $\overline{\Lie}$.

By assumption, the point $x$ is fixed by the flow $\varphi^i_t$. Hence the differential $d_x\varphi^i_t$ as an endomorphism of  $T_x X$ acts as 
$$W \mapsto W + t \cdot d_x f (W)  (V_i)_x$$
by Lemma \ref{calculation}.
For $W:=(\theta_1)_x$, these image vectors span $T_x X$ for general $t$, since the $V_i$'s span $T_x X$ and  $d_x f_i (W) \ne 0$ by assumption. This proves that the fiber $\tilde{L}_x$ contains a generating set.
\end{proof}

As a corollary we have \setcounter{thm}{1}
\begin{thm}
    Let $X$ be a Stein manifold and $(\theta_1,\dots,\theta_n)$ a compatible $n$-tuple. Assume that there are $\C$-complete holomorphic vector fields $V_1,\dots,V_N$ on $X$ which span the tangent bundle $TX$. If there are functions $f_i\in \ker V_i$ such that $f_i(x_0)=0$ and $d_{x_0}f_i(\theta_1)\neq 0$ for some point $x_0\in X$ and $i = 1, \dots, N$, then $X$ has the density property.
\end{thm}

\begin{proof} The fact that the complete fields span all tangent spaces implies holomorphic flexibility. Thus Theorem \ref{theorem:KalKutHolomorph} implies the result. 
\end{proof}

\begin{cor}[Theorem \ref{theorem:SufficientConditionsDP} - algebraic version] \label{cor:theorem2}
    Let $X$ be an affine algebraic manifold and $(\theta_1,\dots,\theta_n)$ a compatible $n$-tuple of algebraic vector fields. Assume that there are $\C$-complete algebraic vector fields $V_1,\dots,V_N$ on $X$ with algebraic flows such that the collection $\{\theta_1,V_1,...,V_N\}$ spans the tangent bundle $TX$. If there exist regular functions $f_i\in \ker V_i$ such that $f_i(x_0)=0$ and $d_{x_0}f_i(\theta_1)\neq 0$ for some point $x_0\in X$ and $i = 1, \dots, N$, then $X$ has the algebraic density property.
\end{cor}

\begin{proof}
    The idea of the proof is similar to the holomorphic case (cf. Lemma \ref{lemma:theorem2}). We only have to show that the flow of $f_iV_i$ is algebraic. Let $\phi_i(x,t)$ denote the flow of $V_i$. Then the flow of $f_iV_i$ is $\phi_i(x,f_i(x)t)$ and thus algebraic.
\end{proof}

\section{Applications}
\subsection{Danielewski surfaces}
Given a polynomial $p\colon\mathbb{C}\to \mathbb{C}$ with simple zeros, we define the variety
\[ D_p:=\{ (x,y,z)\in \mathbb{C}^3: xy=p(z)\} \]
called \textit{Danielewski surface}. This is an algebraic manifold, since $p$ has only simple zeros. Furthermore, it is well-known that $D_p$ has the algebraic density property \cite{KalKutHyperSurface}*{Theorem 1} despite the lack of compatible pairs of complete algebraic vector fields. In this section we give a new, shorter proof of this fact. We show the existence of a compatible 3-tuple and apply the Generalized Kaliman--Kutzschebauch criterion (Theorem \ref{theorem:SufficientConditionsDP}).

The following three are complete algebraic vector fields tangent to the surface $D_p$:
\begin{align*}
\theta_1 = x\deriv{x} - y\deriv{y}\\
\theta_2 = p'(z)\deriv{x} + y\deriv{z}\\
\theta_3 = p'(z)\deriv{y} + x \deriv{z}.
\end{align*}

\begin{lemma}\label{lemma:Danielewski:compatible}
The vector fields $\theta_1,\theta_2$ and $\theta_3$ form a compatible 3-tuple $(\theta_1,\theta_2,\theta_3)$.
\end{lemma}

\begin{proof}
Define the functions $f_x(x,y,z)=x$, $f_y(x,y,z)=y$ and $f_z(x,y,z)=z$. We see that $f_z\in \ker \theta_1$ and
\[ f_z\in \ker \theta_i^2\setminus \ker \theta_i\]
for $i=2,3$. Hence there exists a $\{\theta_1,\theta_2,\theta_3\}$-admissible rooted tree $(T,\pi,\epsilon)$ with root $\theta_1$. More precisely, the map $\epsilon:\mathrm{Edge}(T)\to\mathbb{C}[D_p]$ can be defined by
\[\epsilon(\theta_2,\theta_1):=f_z,\quad \epsilon(\theta_3,\theta_1):=f_z.\]
Moreover, we have $x\in \ker \theta_3, y\in \ker \theta_2$ and $z \in \ker \theta_1$, which shows the existence of a non-zero ideal in
\[ \mathrm{span}\left( \ker\theta_1\cdot \ker \theta_2 \cdot\ker \theta_3\right)\]
and this finishes the proof.
\end{proof}
\begin{lemma}\label{lemma:Danielewski:tangentBundle}
The vector fields $\theta_1,\theta_2$ and $\theta_3$ span the tangent bundle $TD_p$.
\end{lemma}
\begin{proof}
Let $\vec{x}=(x,y,z)\in D_p$ be a point with $p'(z)\neq 0$. Then the vector fields $\theta_2$ and $\theta_3$ span the tangent space $T_{\vec{x}}D_p$. It remains to consider $\vec{y}=(x,y,z)\in D_p$ with $p'(z)= 0$. Observe that we have $p(z)\neq 0$ for such points, since $p$ has only simple zeros. This implies $xy\neq 0$ and, in particular, $x\neq 0$ and $y\neq 0$. Therefore, the vector fields $\theta_1$ and $\theta_2$ span the tangent space $T_{\vec{y}}D_p$.
\end{proof}

The following is \cite{KalKutHyperSurface}*{Theorem 1} in the special case where $n=1$.
\begin{cor}
The Danielewski surface $D_p$ has the algebraic density property.
\end{cor}
\begin{proof}
The conditions of Corollary \ref{cor:theorem2} are satisfied by Lemmas \ref{lemma:Danielewski:compatible}, \ref{lemma:Danielewski:tangentBundle} and \ref{lemma:usefulProperty}. Therefore $D_p$ has the algebraic density property.
\end{proof}

\subsection{Preparation for Gromov--Vaserstein fibers}
In this section we take a look at two factorization problems or rather, at a by-product of their proofs. Every holomorphic mapping $f\colon\C^l \to \SlC$ can be written as a finite product
$$ f= M_1\cdots M_K,$$
where $M_i\colon\C^l \to \SlC$ is a holomorphic mapping of the respective form
$$ \begin{pmatrix}
    1 & & 0 \\ & \ddots & \\ \star & & 1
\end{pmatrix} \quad \text{and}\quad  \begin{pmatrix}
    1 & & \star \\ & \ddots & \\ 0 & & 1
\end{pmatrix}$$
for $i$ odd and $i$ even, respectively. This was proved by Ivarsson and Kutzschebauch \cite{SLnPaper}.

Similarly, every holomorphic mapping
\[ f\colon\C^l\to \SpC \]
can be factorized into a finite product
\[ f = N_1 \cdots N_K, \]
where $N_i\colon\C^l \to \SpC$ is a holomorphic mapping of the respective form
\[ \begin{pmatrix}
    I_n & 0 \\ A & I_n
\end{pmatrix} \quad \text{and}\quad  \begin{pmatrix}
    I_n & A \\ 0&I_n
\end{pmatrix}, \]
where $I_n$ denotes the $n\times n$-identity matrix, $0$ the $n\times n$-zero matrix and $A$ is a symmetric $n\times n$-matrix, i.e. $A^T=A.$ This result can be found in \cite{Schott}.

Both proofs have in common that a suitable polynomial mapping $P=(P_1,\dots,P_k)\colon\C^m\to \C^k$, $m>k$, satisfies \textit{nice enough} properties that justify an application of the Oka principle (for more details on how the Oka principle can be applied here, we recommend the papers mentioned above.)
In the context of this paper, we are only interested in the smooth fibers $P^{-1}(y)$ of these mappings.

It is sometimes useful to distinguish between the symplectic and the special linear case. We shall therefore write $P_{sp}$ and $P_{sl}$ for the respective mappings.

\begin{lemma}
    Each smooth fiber $P^{-1}(y)$ is biholomorphic to a product $ \mathcal{G}\times \C^L,$
    where 
 $$\mathcal{G} = \left\lbrace z=(z_1,\dots,z_m)\in \CC{m}: p(z)=0\right\rbrace $$
 is a smooth variety for some polynomial mapping $p=(p_1,\dots,p_l)\colon\CC{m}\to \CC{l}$, $m>l$, and some positive integer $L$.
 
 We have $l=n$ in the symplectic case and $l=1$ in the case of the special linear group.
\end{lemma}
\begin{proof}
    In the case of the special linear group, this statement is implied by the proof of \cite{SLnPaper}*{Lemma 3.7}. And in the symplectic case, it follows by Lemma 3.14 and Lemma 3.15 in \cite{Schott}.
\end{proof}

\textbf{Convention:} In the same spirit as before, we shall sometimes write $\mathcal{G}_{sp}$ and $\mathcal{G}_{sl}$ for the respective varieties. 
Moreover, since we are only interested in smooth fibers and thus only in smooth varieties $\mathcal{G}$, we refrain from specifying this every time. We will see a classification of the smooth varieties in the next few subsections.

 For $1\leq i_0<\dots<i_l\leq m$, we define the $(l+1)$-tuple of variables $y=(z_{i_0},\dots,z_{i_l})$ and the corresponding vector field
\begin{align}\label{beautyfulVectorField}
    D_y(p):=\det \begin{pmatrix}
    \Deriv{z}{i_0} & \cdots & \Deriv{z}{i_l}\\
    \Deriv{z}{i_0}p_1(z) & \cdots & \Deriv{z}{i_l}p_1(z)\\
    \vdots && \vdots \\
    \Deriv{z}{i_0}p_l(z) & \cdots & \Deriv{z}{i_l}p_l(z)
\end{pmatrix}.
\end{align}

\begin{lemma}
    The vector fields of the form (\ref{beautyfulVectorField}) are tangent to $\mathcal{G}.$
\end{lemma}
\begin{proof}
   Observe that $D_y(p)(p_i)=0$, since the first row equals the $(i+1)$-th row.
\end{proof}

Let $\mathcal{T}=\{(z_{i_0},\dots,z_{i_l}):1\leq i_0<\dots<i_l\leq m\}$ be the set of all $(l+1)$-tuples and ${\mathcal{V}=\{D_y(p): y\in \mathcal{T}\}}$ the collection of vector fields of the form (\ref{beautyfulVectorField}). Furthermore, let $${\Gamma(\mathcal{V})=\{\alpha^*V: \alpha\in \mathrm{Aut}_{hol}(X), V\in \mathcal{V}\}}$$ denote the set of vector fields generated by $\mathcal{V}$.

\begin{prop}\label{prop:spanningVectorFields}
\begin{enumerate}[label=(\roman*)]
    \item There are complete holomorphic vector fields $V_1,\dots,V_N\in \Gamma(\mathcal{V})$ spanning the tangent bundle $T\mathcal{G}_{sl}.$
    \item There are complete holomorphic vector fields $V_1,\dots,V_N \in \Gamma(\mathcal{V})$ spanning the tangent bundle $T\mathcal{G}_{sp}$.
\end{enumerate}
\end{prop}
\begin{proof}
     (i) See Lemma 5.2 and Lemma 5.3 in \cite{SLnPaper}. \\
     (ii) We refer to \cite{Schott}*{Theorem 3.36}, which is actually a very difficult and technical proof, since it involves both abstract arguments and many concrete calculations.    
\end{proof}

\begin{cor}
    A smooth fiber $P^{-1}(y)\cong \mathcal{G}\times \C^L$ has the density property.
\end{cor}
\begin{proof}
    The variety $\mathcal{G}$ is holomorphically flexible by the previous lemma. Then, the claim follows by Theorem \ref{theorem:holFlexDP}.
\end{proof}

   We show that $\mathcal{G}$ also has the density property. The proof is based on an application of Theorem \ref{theorem:SufficientConditionsDP}.
   
\begin{theorem}\label{prop:existenceCompatibleTuple}
    \begin{enumerate}[label=(\roman*)]
        \item There is a compatible $m$-tuple on $\mathcal{G}_{sl} $ for $n\geq 3.$ In particular, $\mathcal{G}_{sl}$ has the density property.
        \item There is a compatible $m$-tuple on $\mathcal{G}_{sp}$ for $n\geq 2.$ In particular, $\mathcal{G}_{sp}$ has the density property.
    \end{enumerate}
\end{theorem}

\begin{proof}
    In the following subsections we will show the existence of a compatible $m$-tuple for each case (Theorem \ref{theorem:SLcompatible} for the special linear and Theorem \ref{theorem:SPcompatible} for the symplectic case). Then the claims follow immediately from Theorem \ref{theorem:SufficientConditionsDP}, Proposition \ref{prop:spanningVectorFields} and Lemma \ref{lemma:usefulProperty} below.
\end{proof}

\begin{lemma}\label{lemma:usefulProperty}
    Let $x\in \mathcal{G}$ and $V\in \Gamma(\mathcal{V})$ with $V_x\neq 0$. Given any tangent vector $${W \in T_x \mathcal{G} \setminus \mathrm{span}(V_x)},$$ there is a holomorphic function $f\in \ker V$ with $f(x)=0$ and $d_x f(W)\neq 0$.
\end{lemma}
\begin{proof}
    Note that it suffices to show the claim for vector fields in $\mathcal{V}$, since the conclusion of the lemma is invariant under holomorphic automorphisms. Let $V=D_y(p)$ be a vector field of the form (\ref{beautyfulVectorField}). Without loss of generality, we may assume that $y=(z_1,\dots,z_{k+1})$ is the corresponding $(k+1)$-tuple, that is, we have
     $$ V = \sum_{i=1}^{k+1} \alpha_i \Deriv{z}{i},$$
     where $\alpha_i$, $i=1,\dots,k+1,$ are regular functions given by
     $$ \alpha_i(z) = V(z_i) = \det \begin{pmatrix}
         \Deriv{z}{1}p_1(z) & \cdots &\Deriv{z}{i-1}p_1(z) & \Deriv{z}{i+1}p_1(z)& \cdots & \Deriv{z}{k+1}p_1(z)\\
         \vdots & &\vdots  & \vdots & & \vdots \\
         \Deriv{z}{1}p_k(z) & \cdots &\Deriv{z}{i-1}p_k(z) & \Deriv{z}{i+1}p_k(z)& \cdots & \Deriv{z}{k+1}p_k(z)
     \end{pmatrix}.$$
     By assumption, we have $V_x\neq 0$. We may therefore assume, again without loss of generality, that $\alpha_{k+1}(x)\neq 0.$ Let $W \in T_x X\setminus \mathrm{span}(V_x)$. 

     The vector fields $V_i:=D_{y_i}(p)$ corresponding to the $(k+1)$-tuples $y_i=(z_1,\dots,z_k,z_i)$, $i=k+2,\dots,m,$ are given by
$$V_i=\alpha_{k+1}\Deriv{z}{i}+ \sum_{j=1}^k \tilde{\alpha}_{ij}\Deriv{z}{j}$$
for some regular functions $\tilde{\alpha}_{ij}, i=k+2,\dots,m, j=1,\dots,k.$ Hence $V_x, (V_{k+2})_x,\dots,(V_m)_x$ form a basis of $T_xX$. Then
$$ W = \lambda V_x + \sum_{i=k+2}^m \mu_i (V_i)_x$$
with $\mu_j\neq 0$ for some $j\in \{k+2,\dots,m\}$. Let $\pi_j$ denote the projection to the $j$-th component. Then we set
$f(z):=\pi_j(z-x)=z_j-x_j$. Observe that $f(x)=0$ and $f\in \ker V$. It remains to show that $d_xf(W)\neq 0$. By construction, $d_xf$ is also the projection to component $j$, that is, $d_xf(W)=\mu_j\neq 0.$
\end{proof}

\subsection{The special linear case}
In this subsection we find two complete holomorphic vector fields $V_1$ and $V_2$ on $\mathcal{G}$ that together form a compatible $2$-tuple $(V_1,V_2)$. For this reason we look more closely at the polynomial mapping $P\colon\C^m\to \C^k.$ For an odd natural number $K$ we denote elements in $\C^{n(n-1)/2}$ as follows
$$  Z_K = (z_{K,21},\dots,z_{K,kl},\dots,z_{K,n(n-1)}),\quad k>l. $$
Then we set
$$ M_K(Z_K) = \begin{pmatrix}
    1 & 0 & \cdots & 0 \\
    z_{K,21} & 1 & \ddots & \vdots \\
    \vdots & \ddots & \ddots & 0 \\
    z_{K,n1} & \cdots & z_{K,n(n-1)} & 1
\end{pmatrix}.$$
For even $K$ we proceed similarly. We only consider the transposed version, that is, we write elements in $\C^{n(n-1)/2}$ as
$$ Z_K = (z_{K,12},\dots,z_{K,kl},\dots,z_{K,(n-1)n}),\quad k<l,$$
and we set
$$ M_K(Z_K) = \begin{pmatrix}
    1 & z_{K,12} & \cdots & z_{K,1n} \\
    0 & 1 & \ddots & \vdots \\
    \vdots & \ddots & \ddots & z_{K,(n-1)n}\\
    0 & \cdots & 0 & 1
\end{pmatrix}.$$
Set $m=\frac{n(n-1)}{2}$. For a fixed natural number $K$, the Gromov--Vaserstein fibration $${P^K=(P^K_1,\dots,P^K_n)\colon(\C^m)^K\to \C^n}$$ is given by
$$ P^K(Z_1,\dots,Z_K) = e_n^T M_1(Z_1)^{-1}\cdots M_K(Z_K)^{-1}.$$

We now present the definition of the variety $\mathcal{G}$.
\begin{defi}
    For $L\geq 2$, $1\leq i\leq n$ and $a\in \C^*$, we define
    $$ \mathcal{G}:=\mathcal{G}_{L,i,a}=\{Z\in (\C^m)^L: P_i^L(Z)=a\}.$$
\end{defi}

\begin{lemma}
    Let $K\geq 3$ and $a=(a_1,\dots,a_n)\in \C^n\setminus \{0\}$. Then we have
    $$ (P^K)^{-1}(a) \cong \mathcal{G}_{K-1,i,a_i}\times \C^l,$$
    for some $1\leq i\leq n,$ and some natural number $l$.
\end{lemma}
We illustrate the idea of the proof with a simple example. A full proof can be found in the proof of \cite{SLnPaper}*{Lemma 3.7}. We consider the case $K=n=3$ and we assume that $a_3\neq 0.$ We have
$$ a = P^3(Z_1,Z_2,Z_3) = P^2(Z_1,Z_2)M_3(Z_3)^{-1}$$
if and only if
\begin{align*}
    P^2(Z_1,Z_2) &= a M_3(Z_3) = \begin{pmatrix}
    a_1 & a_2 & a_3
\end{pmatrix}\begin{pmatrix}
    1 & 0 & 0 \\ z_{3,21} & 1 & 0 \\ z_{3,31} & z_{3,32} & 1
\end{pmatrix} \\ &= \begin{pmatrix}
    a_1 + a_2z_{3,21} + a_3z_{3,31} & a_2 + a_3z_{3,32} & a_3\end{pmatrix}.
\end{align*} 
We can now express the variables $z_{3,31}$ and $z_{3,32}$ in terms of the others, that is,
\begin{align*}
    z_{3,32} &= \frac{1}{a_3}\left(P_2^2(Z_1,Z_2) - a_2\right)\\
    z_{3,31} &= \frac{1}{a_3}\left(P^2_1(Z_1,Z_2)-a_1 - a_2z_{3,21}\right).
\end{align*}
Observe that there are no conditions for $z_{3,21}$, therefore, the fiber $(P^3)^{-1}(a)$ is biholomorphic to
$\{Z\in (\C^3)^2: P^2_3(Z)=a_3\}\times \C$.

\begin{prop}
Let $L\geq 2$ and $a\in \C^*$. Then
\begin{enumerate}[label=(\arabic*)]
    \item $\mathcal{G}_{L,i,a}$ is smooth for $1\leq i<n$,
    \item $\mathcal{G}_{L,n,a}$ is smooth for $a\neq 1.$
\end{enumerate}
\end{prop}

\begin{proof}
    Let $L\geq 2$ and $a\in \C^*$. Then we have
    $$ \mathcal{G}_{L,i,a} \times \C^l \cong (P^{L+1})^{-1}(b)=:\mathcal{F},$$
    for some natural number $l$ and some $b=(b_1,\dots,b_n)\in \C^n.$ In \cite{SLnPaper}*{Remark 4.1} we have a classification of the singular fibers. We distinguish two cases.
    
    For $L$ odd (and obviously $L+1$ even), we have $b_1=\dots=b_{i-1}=0$ and $b_i=a\neq 0.$ Moreover, the fiber $\mathcal{F}$ is regular if and only if $b\neq e_n^T.$ And this proves the claim for odd numbers $L$.

    For $L$ even, we have $b_n=\dots=b_{i+1}=0$ and $b_i=a\neq 0$. Moreover, the fiber $\mathcal{F}$ is regular if and only if $b_n\neq 1.$ And this proves the claim for even numbers $L.$
\end{proof}

\begin{lemma}
    For $L\leq K$, we have
    $$ \Deriv{z}{L,kl} P^K_i = -P_k^L\cdot e_l^TM_L(Z_L)^{-1}\cdots M_K(Z_K)^{-1}e_i$$
\end{lemma}
\begin{proof}
    The product rule implies
    \begin{align*}
    0 &= \Deriv{z}{L,kl} (M_L(Z_L)^{-1}M_L(Z_L)) \\ &= \Deriv{z}{L,kl} M_L(Z_L)^{-1}M_L(Z_L) + M_L(Z_L)^{-1}\underbrace{\Deriv{z}{L,kl} M_L(Z_L)}_{=E_{kl}}, \end{align*}
    where $E_{kl}$ denotes the $n\times n$-matrix having a one at entry $(k,l)$ and zeros elsewhere. Hence
    $$ \Deriv{z}{L,kl} M_L(Z_L)^{-1} = - M_L(Z_L)^{-1}E_{kl}M_L(Z_L)^{-1}.$$
    Another application of the product rule implies
    \begin{align*}
        \Deriv{z}{L,kl} P^K_i &= e_n^T M_1(Z_1)^{-1}\cdots \Deriv{z}{L,kl}M_L(Z_L)^{-1} \cdots M_K(Z_K)^{-1}e_i\\
        &= -\underbrace{e_n^TM_1(Z_1)^{-1}\cdots M_L(Z_L)^{-1}}_{=P^L}\underbrace{E_{kl}M_L(Z_L)^{-1}\cdots M_K(Z_K)^{-1}e_i}_{=e_l^TM_L(Z_L)^{-1}\cdots M_K(Z_K)^{-1}e_i\cdot e_k}\\
        &=- P^L_k \cdot e_l^T M_L(Z_L)^{-1}\cdots M_K(Z_K)^{-1}e_i
    \end{align*}
    and this finishes the proof.
\end{proof}

\begin{lemma}
    For $L<K$ we have $\Deriv{z}{L,kl}P_i^K\not\equiv 0$ on $\mathcal{G}.$ Moreover, we have $\Deriv{z}{K,kl}P_i^K\not\equiv 0$ on $\mathcal{G}$ if and only if ($K$ is odd and $l\geq i$) or ($K$ is even and $l\leq i$).
\end{lemma}
\begin{proof}
First consider the case $L<K$. For a generic point $Z\in \mathcal{G}$ we have $P_k^L(Z)\neq 0$ and $e_l^TM_L(Z_L)^{-1}\cdots M_K(Z_K)^{-1}e_i\neq 0$. Therefore we get $\Deriv{z}{L,kl}P_i^K\not\equiv 0$ on $\mathcal{G}$.

For $L=K$ and $K$ even, we have $e_l^TM_K(Z_K)e_i\equiv 0$ for $l>i$ and thus
$$ \Deriv{z}{K,kl}P_i^K \equiv 0,\quad l>i.$$
At a generic point $Z\in \mathcal{G}$, however, we have $P_k^K(Z)\neq 0$ and $e_l^TM_K(Z_K)e_i\not\equiv 0$ for $l\leq i.$ Therefore $\Deriv{z}{K,kl}P_i^K\not\equiv 0$ for $l\leq i.$

Similarly, for $L=K$ and $K$ odd we can conclude that $\Deriv{z}{K,kl}P_i^K \not\equiv 0$ on $\mathcal{G}$ if and only if $l\geq i.$
\end{proof}

\begin{theorem}\label{theorem:SLcompatible}
    For $K\geq 2$, the complete holomorphic vector fields
    $$V_1 = \frac{\partial P^K_i}{\partial z_{1,n2}}\Deriv{z}{1,n1} - \frac{\partial P_i^K}{\partial z_{1,n1}}\Deriv{z}{1,n2}$$
    and
    $$ V_2 = \frac{\partial P^K_i}{\partial z_{2,2n}}\Deriv{z}{2,1n} - \frac{\partial P_i^K}{\partial z_{2,1n}}\Deriv{z}{2,2n}$$
    build a compatible $2$-tuple $(V_2,V_1)$.
\end{theorem}

\begin{proof}
    Observe that $P_i^K$ is (at most) linear in each variable $z_{L,kl}$. Therefore the two holomorphic vector fields $V_1$ and $V_2$ are $\C$-complete.

    We check the two properties of Definition \ref{def:holomorphicCompatibleTuple}. Note that each variable $z_{L,kl}$ is in the kernel of $V_1$ or $V_2$. Hence we have
        $$ \overline{\mathrm{span} (\ker V_1 \cdot \ker V_2)} = \mathcal{O}(\mathcal{G})$$
   and this implies (1). 

      For property (2), it remains to find a function $f\in \left(\ker V_1^2\setminus \ker V_1 \right)\cap \ker V_2.$ We show that $f=z_{1,n2}$ does the job. Clearly, $f\in \ker V_2$. Moreover, $V_1(f)=-\frac{\partial P^K_i}{\partial z_{1,n1}}\not\equiv 0$ by the above lemma. We prove that
      $$ \Deriv{z}{1,nj} \frac{\partial P_i^K}{\partial z_{1,n2}}=0,\quad 1\leq j \leq n-1.$$
We can write
      $$ M_1(Z_1)= \begin{pmatrix}
          M & 0\\ z^T & 1
      \end{pmatrix} = \begin{pmatrix}
          M & 0 \\ 0 & 1 
      \end{pmatrix}\begin{pmatrix}
          I_{n-1} & 0 \\ z^T & 1
      \end{pmatrix},$$
      for some lower triangular $(n-1)\times (n-1)$-matrix $M$ with $\Deriv{z}{1,nj}M=0$ for $j=1,\dots,n-1$, and $z^T=(z_{1,n1},\dots,z_{1,n(n-1)})$. Since
      $$ M_1(Z_1)^{-1} = \begin{pmatrix}
          I_{n-1} & 0 \\ -z^T & 1
      \end{pmatrix} \begin{pmatrix}
          M^{-1} & 0 \\ 0 & 1
      \end{pmatrix}$$
     we obtain
      $$ \Deriv{z}{1,nj}\Deriv{z}{1,n2} M_1(Z_1)^{-1} = \underbrace{\Deriv{z}{1,nj}\Deriv{z}{1,n2}\begin{pmatrix}
          I_{n-1} & 0 \\ -z^T & 1
      \end{pmatrix}}_{=0}\cdot  \begin{pmatrix}
          M^{-1} & 0 \\ 0 & 1
      \end{pmatrix} = 0$$
      for $j=1,\dots,n-1.$
      By the recursive formula of $P^K$ and the product rule, we conclude that
   $$ \Deriv{z}{1,nj}\frac{\partial P^K_i}{\partial z_{1,n2}} = 0. $$
   This means, in particular, that $V_1(f)\in \ker V_1$, which was to be proved.  
\end{proof}

\subsection{The symplectic case}
Let us start with the definition of a symplectic matrix. Consider the skew-symmetric $2n\times 2n$-matrix
$$ \Omega = \begin{pmatrix}
    0&I_n\\-I_n&0
\end{pmatrix},$$
where $I_n$ denotes the $n\times n$-identity matrix and $0$ the $n\times n$-zero matrix. Then
$$ \SpC = \{M\in \C^{2n\times 2n}: M^T\Omega M=\Omega\}$$
is the symplectic group. As in the previous subsection, we define an elementary mapping $$ M\colon\C^m\to \SpC,$$
where $m= n(n+1)/2$. For a natural number $K$ we write elements in $\C^m$ as follows
$$ Z_K =(z_{K,11},\dots,z_{K,kl},\dots,z_{K,nn}),\quad 1\leq k\leq l\leq n. $$
Observe that 
$$ \psi(Z_K)= \begin{pmatrix}
    z_{K,11}& z_{K,12} & \cdots & z_{K,1n} \\
    z_{K,12} & z_{K,22} & \cdots & z_{K,2n} \\
    \vdots & \vdots & \ddots & \vdots \\
    z_{K,1n} & z_{K,2n} &\cdots & z_{K,nn}
\end{pmatrix}$$
defines an isomorphism $\psi\colon\C^m\to \mathrm{Sym}_n(\C)$, where $\mathrm{Sym}_n(\C)$ denotes the vector space of symmetric $n\times n$-matrices, that is, matrices $A\in \C^{n\times n}$ with $A^T=A$. By abuse of notation, we let $Z_K$ denote both the vector and the matrix. Then we define for even $K$
$$ M_K(Z_K)=\begin{pmatrix}
    I_n & Z_K \\ 0& I_n
\end{pmatrix}$$
and for odd $K$
$$ M_K(Z_K) = \begin{pmatrix}
    I_n & 0\\ Z_K & I_n
\end{pmatrix}.$$
Note that $M_K(Z_K)$ is actually a symplectic matrix. 

For a fixed natural number $K$, the Gromov--Vaserstein fibration $$ P^K=(P^K_1,\dots,P^K_{2n})\colon(\C^m)^K\to \C^{2n}$$ is given by
$$ P^K(Z_1,\dots,Z_K) = e_{2n}^TM_1(Z_1)\cdots M_K(Z_K).$$
We introduce the notation
$$ P_f^K = (P_1^K,\dots,P_n^K),\quad P_s^K=(P^K_{n+1},\dots,P_{2n}^K).$$
\begin{defi}
    Let $a=(a_1,\dots,a_n)\in \C^n\setminus \{0\}$. For $K\geq 2$ even, we define
    $$ \mathcal{G}:=\mathcal{G}_{K,a}:=\{Z\in \C^n\times (\C^m)^{K-1}: P_s^{K}(Z)=a\}.$$
    And for $K\geq 3$ odd, we define
    $$ \mathcal{G}:=\mathcal{G}_{K,a}:=\{Z\in \C^n\times (\C^m)^{K-1}: P_f^{K}(Z)=a\}.$$
\end{defi}

\begin{lemma}
    Let $K\geq 2$ and $a\in \C^n\setminus \{0\}$. Then
$$ \mathcal{G}_{K,a}\times \C^l \cong (P^{K+1})^{-1}(y)$$
for some natural number $l$ and some $y\in \CwO$.
\end{lemma}

\begin{proof}
    Without loss of generality, we assume $K$ to be odd (the even case is symmetrical). Set $y=(a,b)$ for an arbitrary $b\in \C^n.$ We have $a\neq 0$ by assumption, hence $y\in \CwO.$ Now observe that
    \begin{align*}
    \begin{pmatrix}
    a & b 
\end{pmatrix} &= P^{K+1}(Z_1,\dots,Z_{K+1})= P^{K}(Z_1,\dots,Z_{K})M_{K+1}(Z_{K+1}) \\ &= \begin{pmatrix}
    P_f^{K} & P_s^{K}
\end{pmatrix} \Eu{Z_{K+1}}
\end{align*}
and thus
\begin{align*}
    \begin{pmatrix}
    P_f^{K} & P_s^{K}
\end{pmatrix} = \begin{pmatrix}
    a & b 
\end{pmatrix} \Eu{-Z_{K+1}} = \begin{pmatrix}
    a & b - aZ_{K+1}
\end{pmatrix}.
\end{align*}
Since we assume $a\neq 0$, we can rearrange the equation $P_s^K=b-aZ_{K+1}$ in such a way that $n$ of the variables $z_{K+1,kl}$ can be expressed. This leads to
$$ \underbrace{\{Z\in (\C^m)^K:P_f^K(Z)=a\}}_{\tilde{\mathcal{G}}}\times \C^{m-n}\cong (P^{K+1})^{-1}(y).$$
Observe that there are no conditions placed on the variables $z_{1,kl}$ for $1\leq k\leq l < n,$ by definition of $P^{K+1}$ (since we project the first factor $M_1(Z_1)$ to the last row). Therefore we obtain
$$ \tilde{\mathcal{G}} \cong \mathcal{G}\times \C^{m-n}$$
and this finishes the proof. For more details, see \cite{Schott}*{Lemma 3.15}. 
\end{proof}

\begin{prop}
    Let $K\geq 2$ and $a\in \C^n\setminus \{0\}.$ Then $\mathcal{G}_{K,a}$ is smooth if one of the following properties is satisfied
    \begin{enumerate}[label=(\arabic*)]
        \item $K$ is odd,
        \item $K$ is even and $a\neq e_n^T.$
    \end{enumerate}
\end{prop}
\begin{proof}
    A classification of the singular fibers of $P^K$ can be found in sections 3.1.1 and 3.2.1 in \cite{Schott}.

    Suppose $K$ is odd. Then we have
\begin{align}\label{equation:smoothVarieties}
        \mathcal{G}_{K,a}\times \C^l \cong (P^{K+1})^{-1}(y),
    \end{align}
    with $y=(a,b)$ for some $b\in \C^n$. The only singular fiber of $P^{K+1}$ is the one over $y=(0,e_n^T)$. Since we assume $a\neq 0$, the variety $\mathcal{G}_{K,a}$ is smooth.

    For $K$ even, we have Equation (\ref{equation:smoothVarieties}) again, but with $y=(b,a)$ for some $b\in \C^n.$ A fiber $(P^{K+1})^{-1}(b,a)$ is singular if and only if $a= e_n^T.$ And this proves the claim.
\end{proof}

\begin{theorem}\label{theorem:SPcompatible}
    Let $K\geq 2$ and $n\geq 2$. Moreover, let $\mathcal{G}$ be smooth. Then there exists a compatible $k$-tuple on $\mathcal{G}$ for some $k$. In fact, $k$ turns out to be either two or three.
\end{theorem}

\begin{proof}
    Start with the case $n\geq 3$. Consider the $n+1$ variables $x=(z_{1,n1},\dots,z_{1,nn},z_{2,11})$. Then the corresponding vector field (see (\ref{beautyfulVectorField}))
    $$ V:=\begin{cases}
        D_x(P_s^{K}) & K \text{ even},\\
        D_x(P_f^{K}) & K \text{ odd}
    \end{cases}$$
    is complete on $\mathcal{G},$ since it is affine.
    
    Moreover, the holomorphic vector field
    $$ \gamma = (z_{1,n3})^2\Deriv{z}{2,22}- z_{1,n2}z_{1,n3}\Deriv{z}{2,23} + (z_{1,n2})^2\Deriv{z}{2,33}$$
    is complete on $\mathcal{G}$. We claim that $(\gamma, V)$ is a compatible $2$-tuple. Observe that each variable $z_{k,ij}$ is in the kernel of $ \gamma$ or in the kernel of $ V$. Therefore
    $$ \overline{\mathrm{span}(\ker \gamma \cdot \ker V)} = \mathcal{O}(\mathcal{G}).$$
    Furthermore, we have
    $$ f=z_{2,22}\in \ker V\cap \left(\ker \gamma^2\setminus \ker \gamma\right)$$
    and this proves the claim. Note that this argument works for every $K\geq 2$, hence the theorem is true for $n\geq 3$.

    It remains to prove it for $n=2$. Note that in this case, the vector field $\gamma$ from the previous step does not exist. We therefore divide this case into two steps.

    First, we consider $K\geq 3$. We choose again the vector field $V$ from the previous case and choose
    $$ \tilde{\gamma} = (P^2_{n+2})^2 \Deriv{z}{3,11} -(P^2_{n+1}P^2_{n+2}) \Deriv{z}{3,12}+(P^2_{n+1})^2 \Deriv{z}{3,22},$$
    which is also complete on $\mathcal{G}$. Then we conclude that $(\tilde{\gamma},V)$ is a compatible 2-tuple with the same reasoning as before.

    Finally, we consider $K=2$. We are not able to find a compatible $2$-tuple in this case, instead we have a compatible 3-tuple. Here the variety (in the notation of \cite{Sp4}) is 
    $$ \mathcal{G} = \left\lbrace (z_2,z_3,w_1,w_2,w_3)\in \C^5: \begin{pmatrix}
    w_1 & w_2 \\ w_2 & w_3
\end{pmatrix}\begin{pmatrix}
    z_2\\ z_3
\end{pmatrix} = \begin{pmatrix}
    b_1 \\ b_2 \end{pmatrix} \right\rbrace$$
    for some $(b_1,b_2)\in \C^2\setminus \{0\}$.
    The following three vector fields are complete on $\mathcal{G}$:
\begin{align*}
    V_1 &= -z_2w_3 \deriv{z_2} + z_2w_2 \deriv{z_3} + (w_1w_3-w_2^2)\deriv{w_1}\\
    V_2 &= z_3^2\deriv{w_1} -z_2z_3\deriv{w_2} + z_2^2\deriv{w_3}\\
    V_3 &= z_3^2\deriv{z_2} -w_1z_3\deriv{w_2} + (w_1z_2-w_2z_3)\deriv{w_3}
\end{align*}
    Then $(V_1,V_2,V_3)$ is a compatible 3-tuple. To prove this, note that $z_2,z_3\in \ker V_2$, $w_1\in \ker V_3$ and $w_2,w_3\in \ker V_1$. Hence $$\overline{\mathrm{span}(\ker V_1\cdot \ker V_2 \cdot \ker V_3)} = \mathcal{O}(\mathcal{G}).$$
    Moreover, we have $w_2\in \left(\ker V_2^2\setminus \ker V_2\right) \cap \ker V_1$ and $w_2\in \left(\ker V_3^2\setminus \ker V_3\right)\cap \ker V_1$, which means that there is a $\{V_1,V_2,V_3\}$-admissible rooted tree $(T,\pi,\epsilon)$ with root $V_1$. More precisely, the map $\epsilon\colon\mathrm{Edge}(T)\to \mathcal{O}(\mathcal{G})$ can be defined by
    $$ \epsilon(V_2,V_1):= w_2, \quad \epsilon(V_3,V_1):=w_2.$$
    This finishes the proof of the theorem.
\end{proof}

%
%

\end{document}